\documentclass[preprint,12pt]{elsarticle}

\usepackage[inline]{asymptote}
\begin{asydef}
unitsize(1cm);
\end{asydef}

\usepackage{amssymb}
\usepackage{amsmath}
\usepackage{amsthm}
\newtheorem{theorem}{Theorem}[section]
\newtheorem{lemma}[theorem]{Lemma}
\newtheorem{claim}[theorem]{Claim}
\newtheorem{cor}[theorem]{Corollary}
\newtheorem{open}{Open Question}
\newtheorem*{question}{Question}
\newtheorem{statement}[theorem]{Statement}
\newenvironment{proofof}[1]{\begin{proof}[Proof of #1]}{\end{proof}}

\newcommand{\GR}{\operatorname{GR}}

\journal{Discrete Mathematics}

\begin{document}

\begin{frontmatter}



\title{Bipartite and Euclidean Gallai-Ramsey Theory}


\author[mit]{Isabel McGuigan} 
\ead{iem@mit.edu}
\author[mit]{Katherine Pan}

\affiliation[mit]{organization={Department of Mathematics, Massachusetts Institute of Technology},
            addressline={77 Massachusetts Avenue}, 
            city={Cambridge},
            postcode={02139}, 
            state={MA},
            country={United States}}

\begin{abstract}
In this paper, we investigate the following Gallai-Ramsey question: how large must a complete bipartite graph $K_{n_1, n_2}$ be before any coloring of its edges with $r$ colors contains either a monochromatic copy of $G = K_{s,t}$ or a rainbow copy of $H = K_{s,t}$?  We demonstrate that the answer is linear in $r$, and provide more precise bounds for the specific case $s = 2$.  Furthermore, we also consider the following Euclidean Gallai-Ramsey question: given a configuration $H$ in Euclidean space, what is the smallest $n$ such that any $r$-coloring of $n$-dimensional Euclidean space contains a monochromatic or rainbow configuration congruent to $H$? Through a natural translation between edge colorings of the complete bipartite graph $K_{n_1,n_2}$ and colorings of a subset of $(n_1+n_2)$-dimensional Euclidean space, we prove new upper bounds on $n$ for some configurations which can be expressed as Cartesian products of simplices.
\end{abstract}

\begin{keyword}
Euclidean Ramsey \sep Ramsey \sep Gallai-Ramsey \sep bipartite


\end{keyword}

\end{frontmatter}



\section{Introduction}
\label{sec:intro}

\subsection{Gallai-Ramsey Theory}

Given two graphs $G$ and $H$ and an integer $r$, we may ask the following question:

\begin{question}
    What is the least positive integer $n$ such that any coloring of the edges of the complete graph $K_n$ with $r$ colors contains either a rainbow copy of $G$ (having all edges different colors) or a monochromatic copy of $H$ (having all edges the same color)?
\end{question}

The answer to this question is denoted by the Gallai-Ramsey number $\GR_r(G, H)$.  This question is an extension of the corresponding Ramsey-type question, which asks for the minimum $n$ such that any $r$-coloring of $K_n$ contains a monochromatic copy of $H$; indeed, the existence of the Gallai-Ramsey number is guaranteed by the existence of the corresponding Ramsey number $R_r(H)$ \cite{faudree10}.

Like Ramsey numbers, Gallai-Ramsey numbers are difficult to compute, and exact values are only known for a few classes of graphs.  In the earliest work on this subject, Gallai studied graph colorings which contain no rainbow triangle \cite{gallai67}; as a result, much subsequent work has focused on the case where $G = T_3$ is a triangle.  In this case, the Gallai-Ramsey number $\GR_r(T_3, G)$ is known exactly for several small graphs $H$ \cite{bruce19, faudree10, gregory20, gyarfas04,hamlin19,henry19,lei21,magnant22,zou19}. Even when $\GR_r(T_3, G)$ is not known exactly, its asymptotic behavior (as a function of $r$) is well understood \cite{gyarfas10}.  We refer the reader to \cite{FujitaMO14} for a dynamic survey of known Gallai-Ramsey results and discussion of related problems.

The landscape is different when $G$ is not a triangle.  Some work has been done on the case when $G$ is some fixed small graph, such as a triangle with a pendant edge \cite{fujita12,gutierrez17}, a path \cite{thomason06}, or a star \cite{bass16}. If $G$ is allowed to be large, some bounds for $\GR_r(G, H)$ are known when $H$ is a complete graph and $G$ is either a complete graph or a tree (of arbitrary size) \cite{wang23}, but the value of $\GR_r(G, H)$ for other values of $G, H$ is still open.

\subsection{Euclidean Ramsey Theory}

Another extension of the Ramsey problem deals with colorings of Euclidean space rather than graphs.  Given a configuration $K$ of points in Euclidean space and a positive integer $r$, we may ask the following question: 

\begin{question}
    Does there exist an integer $d$ such that every $r$-coloring of $d$-dimensional Euclidean space contains a monochromatic configuration congruent to $K$?
\end{question}

In the case that such an integer $d$ exists, the configuration $K$ is called \textit{Ramsey}, and we write $\mathbb E^d \xrightarrow{r} K$.  The question of which configurations are Ramsey is well studied \cite{franklrodl90,franklrodl86}, and the smallest possible value of $d$ for some fixed small $r$ and small Ramsey configurations is known exactly \cite{bona_euclidean_1993,bona_ramsey-type_1996,cantwell_finite_1996,shader_all_1976,Toth96}.  We refer the reader to the following surveys \cite{Graham82,Graham94,Graham2011OpenPI} for a more comprehensive understanding of known results and open problems in Euclidean Ramsey theory.

\subsection{Euclidean Gallai-Ramsey Theory}

Combining the previous two extensions, we may ask Gallai-Ramsey type questions in a Euclidean setting:

\begin{question}
     Given a positive integer $r$ and two configurations $K$, $K'$, does there exist an integer $d$ such that every $r$-coloring of $d$-dimensional Euclidean space contains a rainbow configuration congruent to $K$ or a monochromatic configuration congruent to $K'$?
\end{question}

Similar the the Euclidean Ramsey case, if such a $d$ exists, then we write $E^{d} \xrightarrow{r} (K, K')_{\GR}$.  Initial work \cite{cheng2023euclidean} \cite{mao2022euclidean} has established upper and lower bounds on $d$ for some specific configurations, including certain simplicies, squares, and rectangles.  In particular, \cite{mao2022euclidean} shows that, if $K$ is a rectangle, then
\begin{align}\label{eq:known_bound}E^{13r+6}\xrightarrow{r}(K, K)_{\GR}.\end{align}
The proof strategy is to construct a mapping from $K_{2r+5, 11r+1}$ to $E^{13r+4}$ such that a $K_{2,2}$ in $K_{2r+5, 11r+1}$ corresponds to a rectangle in $E^{13r+4}$; then, the problem reduces to the graph Gallai-Ramsey problem of proving that any $r$-coloring of $K_{2r+5, 11r+1}$ contains a monochromatic or rainbow $K_{2,2}$.

\subsection{Main results}
In this paper, we extend this method of proof by deriving new graph Gallai-Ramsey results and applying a variation of the map in \cite{mao2022euclidean} to translate these results into Euclidean Gallai-Ramsey results.  The configurations that this method produces results for are described in terms of the Cartesian products of simplices, where for two configurations $K_1 \in E^{n_1}$ and $K_2 \in E^{n_2}$, we define the Cartesian product $K_1\ast K_2 \in E^{n_1+n_2}$ by 
$$K_1\ast K_2 = \{(x_1,\dots,x_{n_1},y_1,\dots,y_{n_2}) \mid (x_1,\dots,x_{n_1}) \in K_1, (y_1,\dots, y_{n_2})\in K_2\}.$$
For example, the Cartesian product of two one-simplices (line segments) is a rectangle, and the Cartesian product of one one-simplex and one two-simplex (a line segment and a triangle) is a right triangular prism.

The Gallai-Ramsey numbers $\GR_r(T_3, K_{s, t})$ are already well-studied \cite{chen18,gutierrez17,liu21,wu19}, but the symmetric case when $H = G = K_{s,t}$ is not, and no nontrivial bounds on $\GR_r(K_{s,t}, K_{s,t})$ are known. 
 Our main graph Gallai-Ramsey result is that the Gallai-Ramsey number of any bipartite graph is linear in $r$:

\begin{theorem}\label{thm:K_st}
    For any pair of positive integers $s, t$, there exists a constant $C_{s,t}$ such that, for $n = C_{s,t}r$, every $r$-coloring of $K_{n,n}$ contains a monochromatic or rainbow $K_{s,t}$.
\end{theorem}

This bound is tight up to the constant factor, as the graph $K_{(t-1)r, (t-1)r}$ on vertices $u_1,$ $\dots$, $u_{(t-1)r}$,$v_1$,$\dots$, $v_{(t-1)r}$ can be $r$-colored with no monochromatic or rainbow $K_{s,t}$ by coloring edge $u_iv_j$ with color $\lceil i/(t-1)\rceil$.  We will determine a constant which is exponential in $t$, but a more careful analysis can determine tighter bounds for the specific cases when $s = 1, 2$:

\begin{theorem}[also proven in \cite{eroh04}]\label{thm:star}
    Let $r$ be an integer, and let $n = (p-1)(q-1) + 1$.  Then, every $r$-coloring of $K_{1, n}$ contains a rainbow copy of $K_{1, p}$ or a monochromatic copy of $K_{1,q}$.
\end{theorem}

\begin{theorem}\label{thm:K_2t}
    Let $r$ be a positive integer, and $t \geq 2$.  Let \begin{align*}
    n_1 &= (6(t-1)-1)r+2,\\
    n_2 &= 2(t-1)\binom{6(t-1)}{2} + 3(t-1)(r+1)+1.
\end{align*}  Then, any coloring of $K_{n_1, n_2}$ with $r$ colors contains a monochromatic $K_{2,t}$ or a rainbow $K_{2,t}$.
\end{theorem}

And the corresponding Euclidean Gallai-Ramsey results are:

\begin{theorem}\label{thm:simplex_product}
    For any $s,t\in\mathbb{N}$ and $a,b\in\mathbb{R}$, let $Q_s$ be the Cartesian product of a regular $(s-1)$-simplex with side length $a$ and $Q_t$ a regular $(t-1)$-simlpex with side length $b$. Then there is a constant $C_{s,t}$ such that $\mathbb{E}^{C_{s,t}r}\xrightarrow{r}(Q_s,Q_t)_{\GR}$.
\end{theorem}

\begin{theorem}\label{thm:simplex}
    For $p,q\in\mathbb{N}$ and $b\in\mathbb{R}$, let $\Delta_{p}$ be a $p$-dimensional regular simplex with side length $b$ and $\Delta_{q}$ a $q$-dimensional regular simplex with side length $b$. Then for any $r\in\mathbb{N}$, $\mathbb{E}^{pq+2}\xrightarrow{r}(\Delta_p;\Delta_q)_{\GR}$.
\end{theorem}

\begin{theorem}\label{thm:prism_polytope}
    For $t,r\in \mathbb{N}$ and $a,b\in \mathbb{R}$, let $Q_{t}$ be the right prismatic polytope obtained by taking the Cartesian product of a line segment with length $a$ and a $(t-1)$-dimensional regular simplex with side length $b$. For any dimension $t$ and number of colors $r$, let 
    \[d = (6(t-1)-1)r+ 3(t-1)(r+1)+2(t-1)\binom{6(t-1)}{2} +3.\]Then, $\mathbb{E}^{d}\xrightarrow{r}(Q_{t},Q_{t})_{\GR}$. 
\end{theorem}

For $t = 2$, Theorem \ref{thm:prism_polytope} implies that if $K$ is a rectangle (i.e., the Cartesian product of a line segment and a one-simplex), then
\[E^{8r+36} \xrightarrow{r} (K, K)_{GR},\]
which is a constant factor improvement over (\ref{eq:known_bound}).

The remainder of this paper is split into three sections. In Section \ref{sec:bipartite}, we present proofs of our bipartite graph Gallai-Ramsey results. In Section \ref{sec:euclidean}, we utilize these bipartite graph Gallai-Ramsey results and prove our Euclidean Gallai-Ramsey results. Finally, we discuss some future work and interesting further problems in Section \ref{sec:future}.

\section{Bipartite Gallai-Ramsey Numbers}
\label{sec:bipartite}

In this section, we present proofs of our main results on bipartite Gallai-Ramsey numbers. 

\subsection{Proof of Theorem \ref{thm:K_st}}

The proof of Theorem \ref{thm:K_st} will require some machinery from extremal graph theory. Define the \textit{Zarankiewicz number} $z(m, n; s, t)$ to be the maximum possible number of edges in a bipartite graph $G = (U \sqcup V, E)$, where $|U| = m$ and $|V| = n$, that does not contain $K_{s,t}$ as a subgraph. The following upper bound for $z(m, n; s, t)$ was first established by K\H{o}vari, S\'{o}s, and Tur\'{a}n:
\begin{theorem}[\cite{bollobas}]\label{lem:zarankiewicz}
    For positive integers $m, n, s, t$:
    $$z(m, n; s, t) < (s-1)^{1/t}(n-t+1)m^{1-1/t}+(t-1)m.$$
\end{theorem}

A slight rearrangement of this bound gives the following lemma:

\begin{lemma}\label{lem:size_bound}
    Let $H$ be a bipartite graph with vertex set $U \sqcup V$ such that at least $m$ vertices in $U$ have degree at least $k$.  If $H$ does not contain $K_{s,t}$ as a subgraph, then
    $$|V| > \left(\frac m {s-1}\right)^{1/t} (k-t+1).$$
\end{lemma}

\begin{proof}
    Because $H$ does not contain $K_{s,t}$ as a subgraph, it can have at most $z(m, |V|; s, t)$ edges.  But because at least $m$ vertices in $U$ have degree at least $k$, $H$ must have at least $mk$ edges.  Therefore, by Theorem \ref{lem:zarankiewicz},
    $$mk \leq z(m, |V|; s, t) <  (s-1)^{1/t}(|V|-t+1)m^{1-1/t}+(t-1)m.$$
    Dividing through by $m\left(\frac {s-1} m\right)^{1/t}$ and then rearranging, we find
    $$\left(\frac m {s-1}\right)^{1/t}(k-t+1) < |V|-t+1 \leq |V|,$$
    as desired.
\end{proof}

We may use Lemma \ref{lem:size_bound} to control the number of edges of the same color that vertices in a monochromatic-$K_{s,t}$-free $K_{n,n}$ may be adjacent to.

\begin{lemma}\label{lem:adj_colors_1}
Suppose $K_{n, n}$ is colored with $r$ colors such that there is no monochromatic copy of $K_{s,t}$.  Let the vertex set of $K_{n, n}$ be $U \sqcup V$.  Suppose positive integers $m, d$ satisfy 
\[\left(\frac m {s-1}\right)^{1/t} (d-t+1) > n.\]
Then, fewer than $mr$ vertices in $U$ and fewer than $mr$ vertices in $V$ are adjacent to at least $d$ vertices of the same color.
\end{lemma}

\begin{proof}
      Fix a color, say, red.  We will show that fewer than $m$ of the vertices in $U$ can be adjacent to at least $d$ red edges.  Indeed, suppose for contradiction that the vertices $u_1, u_2, \dots, u_m$ are each adjacent to at least $d$ red edges.  Let $\tilde U = \{u_1, \dots, u_m\}$ and $\tilde V = \{v_j \in V \mid u_iv_j \text{ is red for some } 1 \leq i \leq m\}$, and consider the subgraph $H$ of $K_{n,n}$ whose vertex set is $\tilde U \sqcup \tilde V$ and whose edges are all the red edges connecting $\tilde U$ to $\tilde V$.  Then, $|\tilde U| = m$, each vertex in $\tilde U$ has degree at least $d$, and $H$ cannot contain $K_{s,t}$ as a subgraph, otherwise we would find a red $K_{s,t}$.  Thus, by Lemma \ref{lem:size_bound},
    \begin{align*}
    \left|\tilde V\right| >\left(\frac m {s-1}\right)^{1/t} (d-t+1)>n.
    \end{align*}
     But this is a contradiction, since $\left|\tilde V \right| \leq |V| = n$.  Therefore, for each of the $r$ colors, fewer than $m$ of the vertices $u_1, \dots, u_n$ can be adjacent to at least $d$ edges of that color; thus, fewer than $mr$ of the vertices in $U$ and fewer than $mr$ of the vertices in $V$ can be adjacent to at least $d$ edges of any color.  The same logic applies in $V$.
\end{proof}

By using Lemma \ref{lem:adj_colors_1} to limit the number of edges of one color that any vertex is adjacent to, we can prove Theorem \ref{thm:K_st} using a simple probabilistic argument.

\begin{proofof}{Theorem \ref{thm:K_st}}
We claim that any $r$-coloring of $K_{n,n}$ for $n = 3(s-1)s^{2t}t^{2t}r$ which does not contain a monochromatic $K_{s,t}$ contains a rainbow $K_{s,t}$.  Suppose that an $r$-coloring of $K_{n,n}$ contains no monochromatic $K_{s,t}$.  Let $d = 4(s-1)s^{2(t-1)}t^{2(t-1)}r$ and $m = (s-1)s^{2t}t^{2t}$.  Then, $m$ and $d$ satisfy
\begin{align*}
\left(\frac m {s-1}\right)^{1/t} (d-t+1) &= \left(\frac {(s-1)s^{2t}t^{2t}} {s-1}\right)^{1/t} (4(s-1)s^{2(t-1)}t^{2(t-1)}r-t+1)\\
&> s^2t^2(3(s-1)s^{2(t-1)}t^{2(t-1)}r)\\
&= 3(s-1)s^{2t}t^{2t}\\
&= n.
\end{align*}
Therefore, by Lemma \ref{lem:adj_colors_2}, fewer than $mr$ vertices in $U$ and fewer than $mr$ vertices in $V$ are adjacent to at least $d$ vertices of the same color. Define
$$U' = \left\{u_i \in U \,\Big|\, u_i \text{ is adjacent to fewer than }d \text{ edges of each color}\right\},$$
$$V' = \left\{v_i \in V \,\Big|\, v_i \text{ is adjacent to fewer than }d \text{ edges of each color}\right\};$$
we have that $|U'|, |V'| > n - mr = 2(s-1)s^{2t}t^{2t}r$.

Now, choose a random collection of $s$ vertices $u_1,\dots, u_s$ in $U'$ and $t$ vertices $v_1,\dots, v_t$ in $V'$. We can bound the probability that any pair of edges among these vertices has the same color:

\begin{claim}
    For any $1\leq i_1, i_2 \leq s$ and $1\leq j_1, j_2 \leq t$, the probability that the edges $u_{i_1}v_{j_1}$ and $u_{i_2}v_{j_2}$ are the same color is at most $\frac{d}{n - mr}.$
\end{claim}
\begin{proof}
    There are two cases, depending on whether or not the two vertices share an edge.

    In the first case, $u_{i_1}v_{j_1}$ and $u_{i_2}v_{j_2}$ do not share a vertex.  Once $u_{i_1}$ and $v_{j_1}$ are fixed, there are $(|U'|-1)(|V'|-1)$ remaining choices for $u_{i_2}$ and $v_{j_2}$.  Since every vertex in $U'$ is adjacent to at most $d$ vertices of the same color, at most $d(|U'|-1)$ of these choices will have the edge $u_{i_2}v_{i_2}$ the same color as $u_{i_1}v_{i_1}$. The probability that these two edges are the same color is thus at most \[\frac{d(|U'|-1)}{(|U'|-1)(|V'|-1)} \leq \frac{d}{n-mr.}\]

    In the second case, $u_{i_1}v_{j_1}$ and $u_{i_2}v_{j_2}$ do share a vertex -- without loss of generality, suppose $i_1 = i_2$. Once $u_{i_1}$ and $v_{j_1}$ are fixed, there are $|V'|-1$ remaining choices for $v_{j_2}$, and at most $d$ of these will have $u_{i_1}v_{i_2}$ the same color as $u_{i_1}v_{i_1}$.  Therefore, the probability that these two edges are the same color is at most
    \[\frac{d}{(|V'|-1)} \leq \frac{d}{n-mr.}\]
    
\end{proof}

There are $st$ edges total among these vertices, so by the union bound, the probability that any pair of them is the same color is at most
\begin{align*}\binom {st} 2 \cdot \frac{d}{n-mr} = \binom {st}{2} \frac{4(s-1)s^{2(t-1)}t^{2(t-1)}r}{2(s-1)s^{2t}t^{2t}r} = \frac{2}{s^2t^2}\binom{st}{2} < 1.\end{align*}
So there is some choice of $u_1,\dots, u_s, v_1,\dots, v_t$ which creates a rainbow $K_{s,t}$, and we are finished.
\end{proofof}

\subsection{Proof of Theorems \ref{thm:star} and \ref{thm:K_2t}}
Theorem \ref{thm:K_st} holds for any values of $s, t$, but for the specific cases $s = 1, 2$, we can determine bounds on $n$ which are much smaller than $s^{2t}t^{2t}r$. The proof of Theorem \ref{thm:star} is a simple application of the pigeonhole principle.

\begin{proofof}{Theorem \ref{thm:star}} \cite{eroh04}
        With $n = (p-1)(q-1) + 1$, let $K_{1,n}$ have vertex set $U \sqcup V$ with $U = \{u\}$ and $V = \{v_1, \dots, v_n\}$.  If $u$ is adjacent to at least $p$ edges with distinct colors, then we have found a rainbow $K_{1, p}$.  Otherwise, the edges adjacent to $u$ have at most $p-1$ colors.  By the pigeonhole principle, there must be at least $q$ edges of the same color, and we have found a monochromatic $K_{1,q}$.  
\end{proofof}

The proof of Theorem \ref{thm:K_2t} is more complex, but still revolves around the pigeonhole principle.  Similar to the proof of Theorem \ref{thm:K_st}, we will begin with a lemma controlling the number of edges of the same color that vertices in a monochromatic-and-rainbow$K_{2,t}$-free $K_{n_1, n_2}$ may be adjacent to.  However, while in the proof of Theorem \ref{thm:K_st} we showed that only a few vertices are adjacent to many edges of the same color, here we will show that almost all vertices are adjacent to lots of edges of the same color.

\begin{lemma}\label{lem:adj_colors_2}
    Suppose $K_{n_1, n_2}$ is colored with $r$ colors such that there is no monochromatic or rainbow copy of $K_{2,t}$.  Let the vertex set of $K_{n_1, n_2}$ be $U \sqcup V$ with $U = \{u_1, \dots u_{n_1}\}, V = \{v_1, \dots, v_{n_2}\}$.  Then, for any $1 \leq i, j \leq n_1$, at least one of $u_i$ and $u_j$ is adjacent to $\frac{n_2-(t-1)(r+1)}{4(t-1)}$ edges of the same color.
\end{lemma}

\begin{proof}
        Fix $1 \leq i, j \leq n_1$.  For each $1 \leq k \leq  n_2$, consider the path $\pi_k = u_iv_ku_j$ (see Figure \ref{fig:paths}). 

        \begin{figure}
            \includegraphics[width=0.3\hsize]{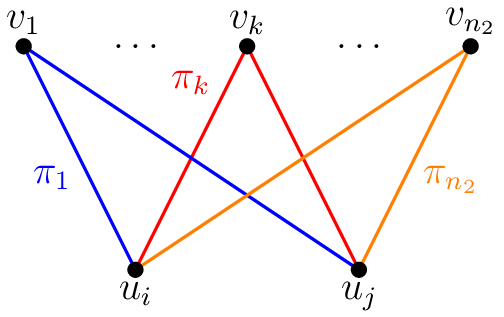}
            \centering
            \caption{Visualization of the paths $\pi_k$}
            \label{fig:paths}
        \end{figure}
        
        Every path $\pi_k$ is either monochromatic or rainbow.  Among these $n_2$ paths, at most $(t-1)r$ of them can be monochromatic: if there were more than $(t-1)r$ monochromatic paths, then by the pigeonhole principle there would have to be $t$ paths $\pi_{k_1}, \dots, \pi_{k_t}$ colored with the same color, but then $u_i, u_j, v_{k_1}, \dots, v_{k_t}$ would form a monochromatic $K_{2,t}$.  Thus, at least $n_2-(t-1)r$ of the paths $\pi_k$ are rainbow.

        Let $\mathcal P$ be a subset of the paths $\pi_1, \dots, \pi_{n_2}$ satisfying the following properties:
        \begin{itemize}
            \item Each $\pi_k \in \mathcal P$ is rainbow.
            \item For each pair of paths $\pi_i, \pi_j \in \mathcal P$, $\pi_i$ and $\pi_j$ share no colors.
            \item $\mathcal P$ is maximal: for each path $\pi_k \not \in \mathcal P$, $\pi_k$ shares a color with some path in $\mathcal P$.
        \end{itemize}

        The set $\mathcal P$ can be built greedily: as long as there's a rainbow path outside of $\mathcal P$ which shares no colors with the paths in $\mathcal P$, add it to $\mathcal P$.  Furthermore, we claim that $|\mathcal P| \leq t-1$: if $|\mathcal P|$ were at least $t$, then it would contain $t$ paths $\pi_{k_1},\dots, \pi_{k_t}$ such that no two paths share any colors; then, the vertices $u_i, u_j, v_{k_1}, \dots, v_{k_t}$ would form a rainbow $K_{2,t}$.  Thus, we must have $|\mathcal P| \leq t-1$.  Since there are at least $n_2-(t-1)r$ rainbow paths, there are at least $n_2-(t-1)(r+1)$ rainbow paths outside of $\mathcal P$.  See Figure \ref{fig:colorful_paths}.
        \begin{figure}
            \includegraphics[width=\hsize]{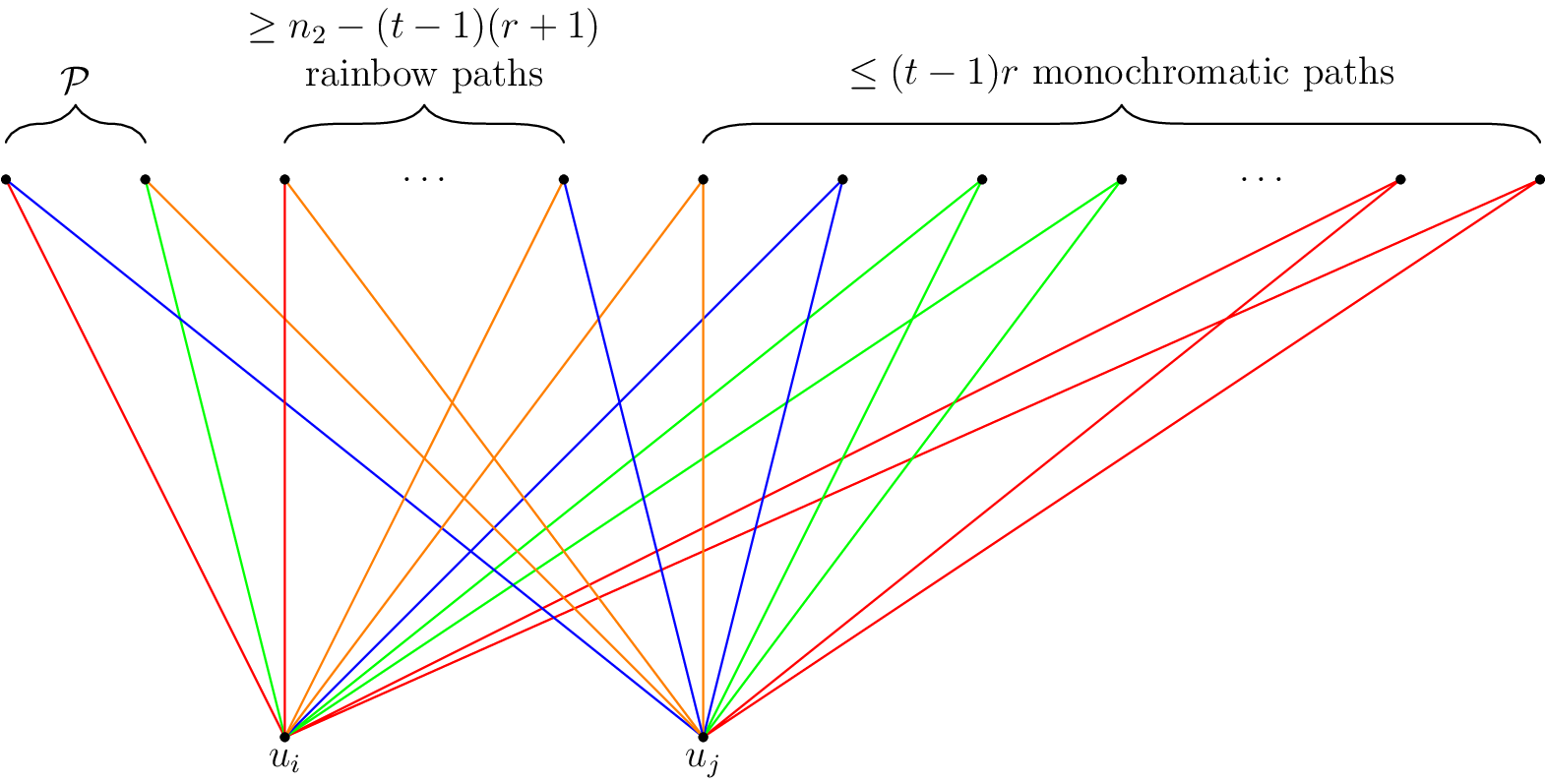}
            \centering
            \caption{All the paths $\pi_k$ adjacent to $u_i$ and $u_j$: at most $(t-1)r$ of them are monochromatic, and at most $t-1$ of them belong to $\mathcal P$.}
            \label{fig:colorful_paths}
        \end{figure}

        The paths in $\mathcal P$ contain $2|\mathcal P|$ unique colors.  Since each path outside of $\mathcal P$ shares a color with some path in $\mathcal P$, and there are at least $n_2-(t-1)(r+1)$ such paths, by the pigeonhole principle, there is some color which appears in at least $\frac{n_2-(t-1)(r+1)}{2|\mathcal P|}$ of these paths. 
        Without loss of generality, suppose this color is red.  Then, for each of the $\frac{n_2-(t-1)(r+1)}{2|\mathcal P|}$ paths containing a red edge, that edge is either connected to $u_i$ or $u_j$; again by the pigeonhole principle, at least $\frac 1 2 \cdot \frac{n-(t-1)(r+1)}{2|\mathcal P|}$ of these red edges are connected to the same vertex.  Thus, either $u_i$ or $u_j$ has at least $\frac{n_2-(t-1)(r+1)}{4|\mathcal P|} \geq \frac{n_2-(t-1)(r+1)}{4(t-1)}$ red edges, as desired.
        
    \end{proof}

\begin{cor}\label{cor:adj_edges}
 Suppose $K_{n_1, n_2}$ is colored with $r$ colors such that there is no monochromatic or rainbow copy of $K_{2,t}$. Let the vertex set of $K_{n_1, n_2}$ be $U \sqcup V$ with $U = \{u_1, \dots u_{n_1}\}, V = \{v_1, \dots, v_{n_2}\}$.  Then there is at most one vertex $u_i$ which is not adjacent to $\frac{n_2-(t-1)(r+1)}{4(t-1)}$ edges of the same color.
\end{cor}

\begin{proof}
    If two such vertices existed, they would contradict Lemma \ref{lem:adj_colors_2}.
\end{proof}

Corollary \ref{cor:adj_edges} implies that in a monochromatic-and-rainbow-$K_{2,t}$-free $K_{n_1, n_2}$, almost all of the vertices are adjacent to a lot of edges of the same color.  Given this, a simple counting argument completes the proof of Theorem \ref{thm:K_2t}.

\begin{proofof}{Theorem \ref{thm:K_2t}}
    With 
    \begin{align*}
    n_1 &= (6(t-1)-1)r+2,\\
    n_2 &= 2(t-1)\binom{6(t-1)}{2} + 3(t-1)(r+1)+1.
\end{align*}
    suppose for a contradiction that there exists an $r$-coloring of $K_{n_1,n_2}$ with no monochromatic or rainbow $K_{2,t}$. Let the vertex set of $K_{n_1, n_2}$ be $U \sqcup V$ with $U = \{u_1, \dots u_{n_1}\}, V = \{v_1, \dots, v_{n_2}\}$. By Corollary \ref{cor:adj_edges}, there is at most one vertex $u_i$ which is not adjacent to $\frac{n_2-(t-1)(r+1)}{4(t-1)}$ edges of the same color, and so at least $n_1-1 = (6(t-1)-1)r+1$ vertices $u_i$ are adjacent to $\frac{n_2-(t-1)(r+1)}{4(t-1)}$ edges of the same color.  By the pigeonhole principle, there is some color -- say, red -- such that at least $6(t-1)$ vertices $u_i$ are adjacent to at least $\frac{n_2-(t-1)(r+1)}{4(t-1)}$ edges of that color.  
    
    Without loss of generality, suppose that $u_1, \dots, u_{6(t-1)}$ are each adjacent to at least $\frac{n_2-(t-1)(r+1)}{4(t-1)}$ red edges, and for $1 \leq i \leq 6(t-1)$ let $A_i = \{v_k \mid 1 \leq k \leq n, \text{edge } u_iv_k\text{ is red}\}$. If $|A_i \cap A_j| \geq t$ for any $1 \leq i, j \leq 6(t-1)$, then $u_i, u_j$, and any $t$ vertices in $A_i \cap A_j$ will form a monochromatic red $K_{2,t}$, so we must have $|A_i\cap A_j|\leq t-1$ for each $i, j$. Using the principle of inclusion and exclusion, this tells us that 
    \begin{align*}
    \left|\bigcup_{i=1}^{6(t-1)}A_i\right| &\geq \sum_{i=1}^{6(t-1)} |A_i| - \sum_{1\leq i < j \leq 6(t-1)} |A_i \cap A_j| \\
    &\geq 6(t-1)\cdot \frac{n_2-(t-1)(r+1)}{4(t-1)} - \binom{6(t-1)}{2}(t-1)\\
    &= \frac 3 2 \cdot \left[2(t-1)\binom{6(t-1)}{2} + 2(t-1)(r+1)+1\right] - \binom{6(t-1)}{2}(t-1)\\
    &=2(t-1)\binom{6(t-1)}{2} + 3(t-1)(r+1) + \frac 3 2\\
    &= n_2 + \frac 1 2.
    \end{align*}
    But this is a contradiction, since each $A_i$ is a subset of $V$, and thus $\left|\bigcup_{i=1}^{6(t-1)}A_i\right| \leq |V| = n_2$.  Therefore, the coloring must have at least one monochromatic or rainbow $K_{2,t}$, completing the proof. 
\end{proofof}

\section{Euclidean Gallai-Ramsey Numbers}
\label{sec:euclidean}

By mapping the edge set of a complete bipartite graph into a subset of Euclidean space, we are able to translate our upper bounds on the graph Gallai-Ramsey numbers obtained in the previous section to upper bounds on the Euclidean Gallai-Ramsey numbers for certain configurations.  The map we will use is a slight generalization of a related map found in \cite{mao2022euclidean}.  To define it, we first define the set of points
\[Q_{s,a} = \left\{\left(\underbrace{0,0,\dots,0}_{i-1},\frac{a}{\sqrt{2}},\underbrace{0,0,\dots,0}_{s-i}\right)\in\mathbb{E}^{s}:1\leq i \leq s\right\}\]
In other words, $Q_{s,a}$ is the set of points in $s$-dimensional Euclidean space with exactly one nonzero coordinate equal to $\frac{a}{\sqrt 2}$.  The points in $Q_{s,a}$ are the vertices of a regular $(s-1)$-dimensional simplex with side length $a$, and thus the Cartesian product $Q_{s,a}\ast Q_{t,b}$ is the Cartesian product of a regular $(s-1)$-dimensional simplex of side length $a$ and a regular $(t-1)$-dimensional simplex of side length $b$.

However, recall that the Cartesian product of two point sets $K_1, K_2 \in E^{n_1}, E^{n_2}$ is defined by
$$K_1\ast K_2 = \{(x_1,\dots,x_{n_1},y_1,\dots,y_{n_2}) \mid (x_1,\dots,x_{n_1}) \in K_1, (y_1,\dots, y_{n_2})\in K_2\}.$$
Therefore, $Q_{s,a}\ast Q_{t,b}$ can also be described as the set of all points in $E^{s+t}$ with exactly one of the first $s$ coordinates equal to $\frac{a}{\sqrt 2}$ and one of the last $t$ coordinates equal to $\frac{b}{\sqrt{2}}$.  Denote by $x_{s, t, a, b, i, j}$ the point in $Q_{s,a}\ast Q_{t, b}$ which has $\frac{a}{\sqrt{2}}$ in the $i$-th coordinate and $\frac{b}{\sqrt{2}}$ in the $(s+j)$-th coordinate; that is,
\[x_{s,t,a,b,i,j} =\left(\underbrace{0,0,\dots,0}_{i-1},\frac{a}{\sqrt{2}},\underbrace{0,0,\dots,0}_{s+j-i-1},\frac{b}{\sqrt{2}},\underbrace{0,0,\dots,0}_{t-j}\right).\]

We now define the map $\phi_{s,t,a,b}: E \to Q_{s,a}\ast Q_{t,b}$ by
    \[\phi_{s,t,a,b}\left(u_iv_j\right)=x_{s,t,a,b,i,j}\] 

    For concision, we omit $s,t,a,b$ in subscripts when they are clear from context. Note that $\phi$ is a bijection between $E$ and $Q_{s,a}\ast Q_{t,b}$.  Furthermore, the image of a subgraph $K_{k, l}$ of $K_{s, t}$ under $\phi$ is itself the Cartesian product of a regular $(k-1)$-dimensional simplex of side length $a$ and a regular $(l-1)$-dimensional simplex of side length $b$ (see Figure \ref{fig:phi}).
\begin{figure}
    \centering
    \includegraphics[width=0.7\hsize]{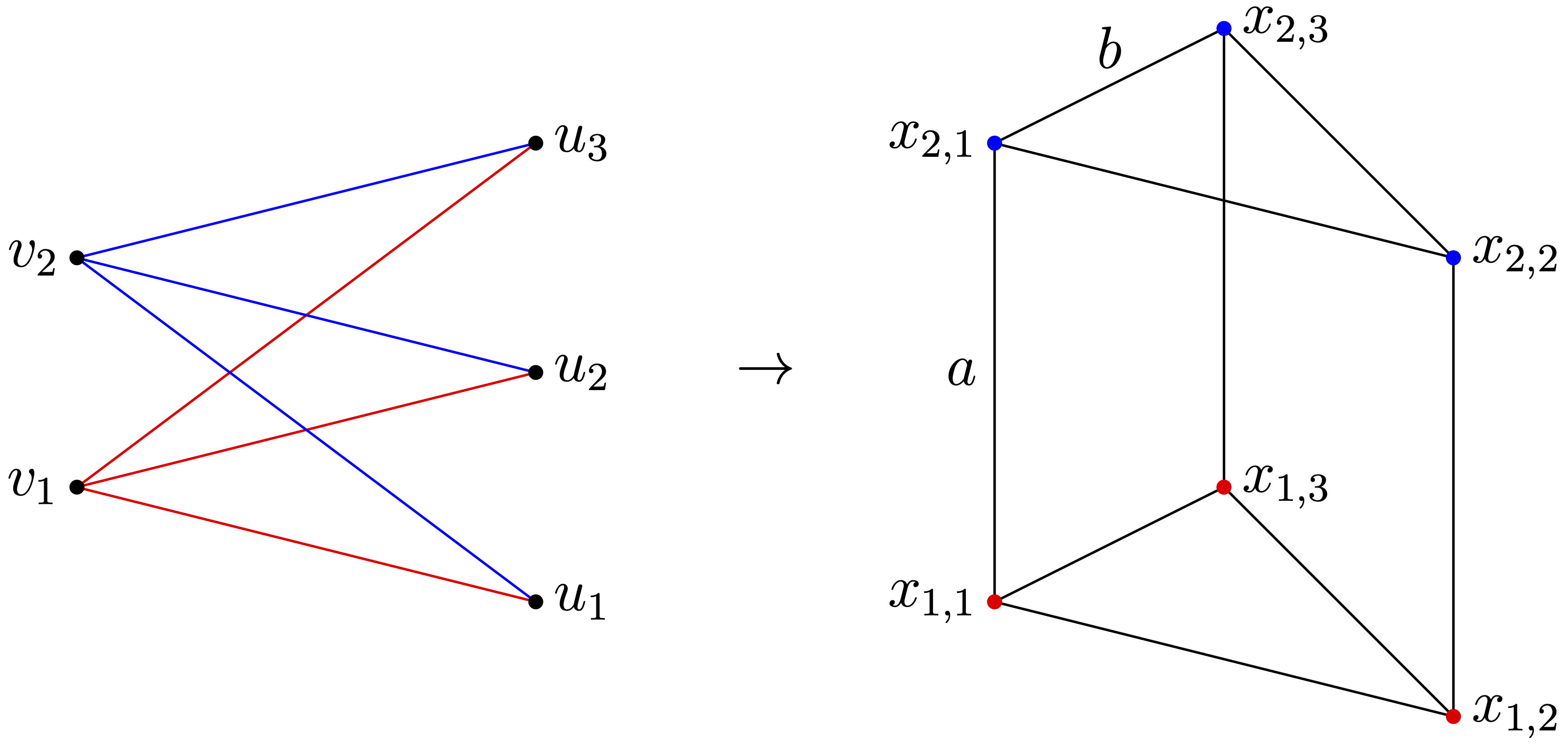}
    \caption{A visualization of $\phi$ from the edges of $K_{2,3}$  to an equilateral triangular prism (the Cartesian product of a regular 1-simplex with side length $a$ and a regular 2-simplex with side length $b$). For example, $\phi$ maps $(v_1,u_1)$ to $x_{1,1} = (\frac{a}{\sqrt{2}}, 0, \frac{b}{\sqrt{2}},0,0)$. Note that the figure on the right is a 3-dimensional projection of the resulting prism, which lies in 5-dimensional space.}
    \label{fig:phi}
\end{figure}
The map $\phi$ gives a correspondence between edge colorings of $K_{s,t}$ and point colorings of $Q_{s,a}\ast Q_{t,b}$, which allows us to translate graph Gallai-Ramsey results to Euclidean Gallai-Ramsey results via the following lemma.

\begin{lemma}\label{translation}
    Suppose that every $r$-coloring of $K_{n_1,n_2}$ contains a rainbow $K_{s_1,t_1}$ or a monochromatic $K_{s_2,t_2}$. Let $S_1 = \Delta_{s_1-1}\ast\Delta_{t_1-1}$ be the Cartesian product of a regular $(s_1-1)$-simplex of side length $a$ and a regular $(t_1-1)$-simplex of side length $b$, and define $S_2 = \Delta_{s_2-1}\ast\Delta_{t_2-1}$ similarly.  Then, every $r$-coloring of $\mathbb E^{n_1+n_2}$ contains a rainbow $S_1$ or a monochromatic $S_2$.
\end{lemma}

\begin{proof}
    Consider the subset $Q_{n_1,a}\ast Q_{n_2,b} \subset \mathbb E^{n_1+n_2}.$ An $r$-coloring of $Q_{n_1,a}\ast Q_{n_2,b}$ determines an $r$-coloring of $K_{n_1, n_2}$ by coloring edge $u_i, v_j$ with the color of the point $\phi(u_i, v_j)$. By assumption, this $r$-coloring contains a rainbow $K_{s_1,t_1}$ or a monochromatic $K_{s_2,t_2}$. The image of $K_{s_1,t_1}$ under $\phi$ is congruent to $S_1 = \Delta_{s_1-1}\ast\Delta_{t_1-1}$, and the image of $K_{s_2, t_2}$ under $\phi$ is congruent to $S_2 = \Delta_{s_2-1}\ast\Delta_{t_2-1}$. Thus, if $K_{n_1, n_2}$ contains a rainbow $K_{s_1,t_1}$, then $Q_{n_1,a}\ast Q_{n_2,b}$ contains a rainbow $S_1$, and if $K_{n_1,n_2}$ contains a rainbow $K_{s_2,t_2}$, then $Q_{n_1,a}\ast Q_{n_2,b}$ contains a monochromatic $S_2$.  Since $Q_{n_1,a}\ast Q_{n_2,b}$ lies in $\mathbb E^{n_1+n_2}$, any $r$-coloring of $\mathbb E^{n_1+n_2}$ must contain a rainbow $S_1$ or a monochromatic $S_2$.
\end{proof}

Theorems \ref{thm:simplex}, \ref{thm:prism_polytope}, and \ref{thm:simplex_product} now follow immediately from Theorems \ref{thm:star}, \ref{thm:K_2t}, and \ref{thm:K_st} via Lemma \ref{translation}.

\section{Future Work}
\label{sec:future}

Throughout this paper, we leveraged the a connection between graphs and Euclidean space to obtain Euclidean Gallai-Ramsey numbers.  However, by focusing on a finite, discrete set of Euclidean space, this approach necessarily loses out on the geometric power that other approaches to Euclidean Ramsey and Gallai-Ramsey theory provide.  In the process of obtaining the main results presented in this paper, the authors note a few interesting questions of Euclidean Gallai-Ramsey theory that do not seem amenable to the methods presented here.

The first question asks about the distinction between Ramsey and Gallai-Ramsey configurations:

\begin{open}
    Does there exist a configuration that is not Ramsey, but is Gallai-Ramsey?
\end{open}

Clearly, all Ramsey configurations are Gallai-Ramsey, but it is not clear whether the converse is true.  An equivalent formulation of this question asks about Cartesian products of Gallai-Ramsey configurations.  It is known that the Cartesian product of two Ramsey configurations is also Ramsey \cite{franklrodl90}. We may ask if the analogous statement holds for Gallai-Ramsey configurations: 
\begin{statement}\label{pt}
    If $K_1$ and $K_2$ are Gallai-Ramsey, then the Cartesian product $K_1 \ast K_2$ is also Gallai-Ramsey.
\end{statement} 

\begin{theorem}
    Statement \ref{pt} is true if and only if all Gallai-Ramsey configurations are Ramsey.
\end{theorem}

\begin{proof}
    The first direction follows immediately from the product theorem for Ramsey configurations: suppose all Gallai-Ramsey configurations are Ramsey. Then if $K_1$ and $K_2$ are Gallai-Ramsey, they're Ramsey, so $K_1\ast K_2$ is Ramsey, and hence Gallai-Ramsey.  Hence, if all Gallai-Ramsey configruations are Ramsey, Statement \ref{pt} holds for Gallai-Ramsey configurations.
    
    In the other direction, suppose Statement \ref{pt} is true for Gallai-Ramsey configurations, and let $K_1$ be a Gallai-Ramsey configuration.  We will show that $K_1$ is Ramsey.  For any integer $r$, let $K_2$ be a regular $r$-dimensional simplex with $r+1$ vertices.  Because $K_2$ is Ramsey, it's Gallai-Ramsey, so by assumption, $K_1 \ast K_2$ is Gallai-Ramsey.  Thus, there's some $n$ such that every coloring of $\mathbb E^n$ contains a monochromatic or rainbow copy of $K_1 \ast K_2$ with $r$ colors. But $K_1\ast K_2$ contains $|K_1||K_2| = |K_1|(r+1) > r$ points, so no coloring of $\mathbb E^n$ with $r$ colors can contain a rainbow copy of $K_1\ast K_2$.  Hence, every $r$-coloring of $\mathbb E^n$ contains a monochromatic $K_1 \ast K_2$, and hence a monochromatic $K_1$.  Therefore, $K_1$ is Ramsey. 
\end{proof}

 One plausible candidate for a set which is Gallai-Ramsey but not Ramsey is $l_3$, the set of three co-linear points with pairwise distances 1, 1, and 2.  It is known that any Ramsey configuration must be embeddable in a sphere \cite{erdos1973}, and since this is not true of $l_3$, it is not Ramsey. It is not known whether $l_3$ is Gallai-Ramsey.  However, if we restrict ourselves to \textit{spherical} colorings of $\mathbb E^n$ (colorings for which all points of the same magnitude are the same color), then it is known that every spherical coloring of $\mathbb E^n$ for $n \geq 2$ contains a monochromatic or a rainbow $l_3$.  This makes the general question quite tantalizing.
\begin{open}
    Is $l_3$ Gallai-Ramsey?
\end{open}
In a separate vein, we note that most of the Euclidean Gallai-Ramsey results obtained in this paper involve a linear or exponential dependency on $r$. However, for some configurations, like a square, it is known that there is an absolute constant $C$ such that any coloring of $\mathbb E^C$ contains a monochromatic or rainbow copy of the configuration, regardless of the number of colors \cite{cheng2023euclidean}.  We are curious if there exist bounds for the configurations studied in this paper that have no dependency on the number of colors.

\begin{open}
    For fixed integers $s, t$, does there exist an absolute constant $C$ such that any coloring of $\mathbb E^C$ contains a monochromatic or a rainbow copy of $\Delta_s\ast \Delta_t$?
\end{open}

\section{Acknowledgements}
\label{sec:acknowledgements}

The authors are grateful to Lisa Sauermann and Zixuan Xu for their mentorship and helpful suggestions throughout the project, as well as to David Jerison and the rest of the MIT SPUR faculty for organizing the project.
\bibliographystyle{elsarticle-num} 
\bibliography{cas-refs}


\end{document}